\documentclass[reqno]{amsart}
\usepackage{amssymb,amsmath,amsthm,mathrsfs, array}
\usepackage{amsfonts,latexsym,amsxtra,amscd}

\newtheorem{thm}{Theorem}[section]

\newtheorem{cor}[thm]{Corollary}
\newtheorem{lem}[thm]{Lemma}

\newtheorem{rem}[thm]{Remark}

\numberwithin{equation}{section}

\begin{document}

\title[Quantitative recurrence and homogeneous self-similar sets]
{Quantitative recurrence properties and homogeneous self-similar sets}

\author{Yuanyang Chang}

\author{Min Wu}

\author{Wen Wu$^{\ast}$}
\address{Department of Mathematics, South China University of Technology, Guangzhou, 510640, P. R. China}
\email{changyy@scut.edu.cn, wumin@scut.edu.cn, wuwen@scut.edu.cn}
\thanks{$^{\ast}$Corresponding author.}

\date{}


\begin{abstract}
Let $K$ be a homogeneous self-similar set satisfying the strong separation condition. This paper is concerned with the quantitative recurrence properties of the natural map $T: K\rightarrow K$ induced by the shift. Let $\mu$ be the natural self-similar measure supported on $K$. For a positive function $\varphi$ defined on $\mathbb{N}$, we show that the $\mu$-measure of the following set
\begin{equation*}
  R(\varphi):=\{x\in K: |T^n x-x|<\varphi(n) \; \text{for infinitely many} \; n\in\mathbb{N}\}
\end{equation*}
is null or full according to convergence or divergence of a certain series. Moreover, a similar dichotomy law holds for the general Hausdorff measure, which completes the metric theory of this set.
\end{abstract}

\maketitle

\section{Introduction}
Let $(X, d)$ be a separable metric space, $\mathcal{B}$ the Borel $\sigma$-algebra, $T: X\rightarrow X$ a Borel measurable map and $\mu$ a $T$-invariant Borel probability measure on $X$. The quintuple $(X, \mathcal{B}, \mu, d, T)$ is called a metric measure-preserving system (MMPS). The famous Poincar\'{e} Recurrence Theorem implies that $\mu$-almost every $x\in X$ is recurrent in the sense that
\begin{equation}\label{Poincare}
  \liminf_{n\rightarrow\infty} d(T^n x, x)=0.
\end{equation}
That is to say, for $\mu$-almost every $x\in X$, the orbit $\{T^n x\}_{n\geq 0}$ returns to shrinking neighborhoods of the initial point $x$ infinitely often. However, \eqref{Poincare} is qualitative in nature which does not tell anything about the rate at which a generic orbit comes back to the start point or in what manner the neighborhoods of the start point can shrink. This motivates many authors to investigate the so-called quantitative recurrence properties. For instance, the dynamical Borel-Cantelli lemma \cite{CK01, FMP07}, the first return time \cite{BS01}, the shrinking target problems \cite{HV95, LWW14}, and etc. Among them, Boshernitzan \cite{Bos93} first quantified the recurrence rate of generic orbits for general MMPSs.
\begin{thm}\label{Bosh-thm}\emph{(see \cite{Bos93})}
  Let $(X, \mathcal{B}, \mu, d, T)$ be an MMPS. Assume that, for some $\alpha>0$, the $\alpha$-dimensional Hausdorff measure $\mathcal{H}^{\alpha}$ is $\sigma$-finite
on $X$. Then for $\mu$-almost every $x\in X$,
\begin{equation}\label{Bosh-result-1}
  \liminf_{n\rightarrow\infty} n^{\frac{1}{\alpha}}d(T^n x, x)<\infty.
\end{equation}
If, moreover, $\mathcal{H}^{\alpha}(X)=0$, then for $\mu$-almost every $x\in X$,
\begin{equation}\label{Bosh-result-2}
  \liminf_{n\rightarrow\infty} n^{\frac{1}{\alpha}}d(T^n x, x)=0.
\end{equation}
\end{thm}

Supposing $X$ is a measurable subset of a Euclidean space, Barreira and Saussol \cite{BS01} related the recurrence rate of a generic point $x$ to the lower pointwise dimension $\underline{d}_{\mu}(x)$ of $\mu$, defined as
\begin{equation*}
  \underline{d}_{\mu}(x):=\liminf_{r\rightarrow 0}\frac{\log\mu(B(x, r))}{\log r}.
\end{equation*}
\begin{thm}\label{Bar-Sau}\emph{(see \cite{BS01})}
   If $T: X\rightarrow X$ is a Borel measurable map on a measurable subset $X\subset\mathbb{R}^n$ for some $n\in\mathbb{N}$, and $\mu$ is a $T$-invariant Borel probability measure on $X$. Then for $\mu$-almost every $x\in X$, we have
\begin{equation*}
  \liminf_{n\rightarrow\infty} n^{\frac{1}{\alpha}}d(T^n x, x)=0\;\;\text{for any}\; \alpha>\underline{d}_{\mu}(x).
\end{equation*}

\end{thm}

Boshernitzan's result \eqref{Bosh-result-1} implies that, generically,
\begin{equation*}
 d(T^n x, x)<c\cdot n^{-1/\alpha}\;\;\text{for infinitely many} \; n\in\mathbb{N}
\end{equation*}
for a uniform constant $c>0$. Theorem \ref{Bar-Sau} indicates that the recurrence rate could be related to some indicator functions of $x$. Moreover, if the function on the right hand side decays faster, the set of points with such a recurrence rate becomes smaller in the sense that it can not support any $T$-invariant Borel probability measures. So, naturally, one would like to measure the size of the set of recurrent points when the recurrence rate $c\cdot n^{-1/\alpha}$ is replaced by a general non-increasing function. 

Let $(X, \mathcal{B}, \mu, d, T)$ be an MMPS and $\varphi$ be a positive function defined on
$\mathbb{N}\times X$. Consider the following dynamically defined limsup set
\begin{equation*}
   R(T, \varphi)=\{x\in X:\; d(T^n x, x)<\varphi(n, x)\;\text{for infinitely many} \; n\in\mathbb{N}\}.
\end{equation*}
Tan and Wang \cite{TW11} calculated the Hausdorff dimension of $R(T_{\beta}, \varphi)$ for the beta dynamical system $([0, 1], T_{\beta})$ with $\beta>1$. Later, Seuret and Wang \cite{SW15} generalized their results to conformal iterated function systems. However, as far as we know, the Hausdorff measure side of $R(T, \varphi)$ for general MMPSs and general $\varphi$ is rarely known. For the purpose of comparison, when we require $\{T^n x\}_{n\geq 1}$ returns to the neighborhoods of a chosen point $x_0$ rather than the initial point $x$, the problem becomes the so-called shrinking target problem (STP). One can refer to \cite{CK01, FMP07} for the measure aspect of STP, and to \cite{HV95, LWW14, WW15} for the dimension aspect.

In this paper, we consider the natural map on a homogeneous self-similar set satisfying a certain separation condition and aim to provide a complete metric theory of the recurrence properties for this concrete system. As we shall see, this is closely related to a folklore problem raised by Mahler (see \cite{Mah84}, Sect. 2): \emph{How close can irrational elements of Cantor's set be approximated by rational numbers?}

To begin with, let $\mathcal{I}=\{\phi_j(x)=\rho x+a_j\}_{j=1}^L$ be a linear iterated function system on $[0, 1]$ (without loss of generality, we assume $0\leq a_1<\cdots<a_L\leq 1-\rho$). It is well known that there exists a unique non-empty compact subset $K\subset [0, 1]$, called the attractor of $\mathcal{I}$, such that $K=\bigcup_{j=1}^L\phi_j(K)$. We further assume the strong separation condition holds for $\mathcal{I}$, that is, the pieces $\left\{\phi_j(K)\right\}_{j=1}^L$ are pariwise disjoint. As a consequence, $\rho\in (0, 1/L)$.
Denote $\Lambda=\{1,\cdots, L\}$. Let $\Lambda^*=\bigcup_{n\geq 1}\Lambda^n$ and $\Lambda^{\mathbb{N}}$ be the space of finite words and infinite words over $\Lambda$ respectively. For any $n\geq 1$ and $\epsilon\in\Lambda^n$, let $\phi_{\epsilon}=\phi_{\epsilon_1}\circ\cdots\circ\phi_{\epsilon_n}$. If $\epsilon\in\Lambda^{\mathbb{N}}$, we denote by $\epsilon|_1^n$ the word
$(\epsilon_1,\cdots,\epsilon_n)$. For each $\epsilon\in\Lambda^{\mathbb{N}}$, the sequence of compact sets $\phi_{\epsilon|_1^n}([0, 1])$ is nested.
Thus the set $\bigcap_{n\geq 1}\phi_{\epsilon|_1^n}([0, 1])$ is a singleton, which we denote by $\pi(\epsilon)$. Therefore,
\begin{equation*}
   K=\pi(\Lambda^{\mathbb{N}})=\bigcup_{\epsilon\in\Lambda^{\mathbb{N}}}\bigcap_{n\geq 1}\phi_{\epsilon|_1^n}([0, 1])
    =\bigcap_{n\geq 1}\bigcup_{\epsilon\in\Lambda^n}\phi_{\epsilon}([0, 1]).
\end{equation*}
The map $\pi: \Lambda^{\mathbb{N}}\rightarrow K$ is a coding map and we call $\epsilon\in\Lambda^{\mathbb{N}}$ a coding of $x\in K$ if $\pi(\epsilon)=x$. Because of the strong separation condition, there are at most countably many points $x\in K$ with multiple codings. Since a countable set is negligible in the sense of Hausdorff measure and Hausdorff dimension, we may assume that each $x\in K$ has a unique coding.

Now we are ready to define the map $T: K\rightarrow K$ as follows: for each $x\in K$ with the coding $\epsilon$, let $T x=\pi(\sigma(\epsilon))$, where $\sigma$ denotes the shift map on the symbolic space $\Lambda^{\mathbb{N}}$. Under the strong separation condition, $K$ is a set of Lebesgue measure zero and has Hausdorff dimension $\dim_H K=\frac{\log L}{-\log\rho}$. Moreover, there exists a Borel probability measure $\mu$ supported on $K$ such that $\mu=\frac{1}{L}\sum_{j=1}^L {\phi_j}\mu$, where $\phi_j \mu$ denotes the push-forward measure of $\mu$ by $\phi_j$, that is,
$\phi_j \mu(A)=\mu\left(\phi_j^{-1}(A)\right)$ for all Borel subsets $A$. In fact, $\mu$ is equal to $\mathcal{H}^{\gamma}|_K$, the $\gamma$-dimension Hausdorff measure restricted to $K$ (see \cite{Fal90, Mat95}). Then $(K, \mathcal{B}, \mu, |\cdot|, T)$ constitutes an MMPS, where $|\cdot|$ denotes the Euclidean metric and $\mathcal{B}$ the Borel $\sigma$-algebra on $K$.

Let $\varphi: \mathbb{N}\rightarrow\mathbb{R}^+$ be a positive function. Set
\begin{equation}\label{recurrent}
  R(\varphi):=\{x\in K: |T^n x-x|<\varphi(n) \; \text{for infinitely many} \; n\in\mathbb{N}\}.
\end{equation}
We obtain the following dichotomy law for the $\mu$-measure of $ R(\varphi)$.
\begin{thm}\label{dict-lebesgue}
   Let $(K, \mathcal{B}, \mu, |\cdot|, T)$ be the MMPS defined as above and $\varphi:\mathbb{N}\rightarrow\mathbb{R}^+$.
   Let $\gamma:=\dim_H K=\frac{\log L}{-\log\rho}$. Then \par
\smallskip
   $\mu(R(\varphi))=\left\{
  \begin{array}{ll}
   0, &\text{if}\; \sum_{n\geq 1} \varphi^{\gamma}(n)<\infty;\\
   1, &\text{if}\; \sum_{n\geq 1} \varphi^{\gamma}(n)=\infty.
  \end{array}
\right.$
\end{thm}

\smallskip
Let $f$ be a doubling dimension function and let $\mathcal{H}^f$ denote the corresponding Hausdorff $f$-measure (see Sect. 2 for the definitions). Then a similar criterion holds for $\mathcal{H}^f$.
\begin{thm}\label{dict-hausdorff}
   Let $(K, \mathcal{B}, \mu, |\cdot|, T)$ be the MMPS defined as above and $\varphi:\mathbb{N}\rightarrow\mathbb{R}^+$. Let $f$ be a doubling dimension function with $r^{-\gamma}f(r)$ being increasing as $r\rightarrow 0$. Then \par
\smallskip
   $\mathcal{H}^f (R(\varphi))=\left\{
  \begin{array}{ll}
   0, &\text{if}\; \sum_{n\geq 1} f\big(\rho^n\varphi(n)\big)\cdot \rho^{-\gamma n}<\infty;\\
   \mathcal{H}^f (K), &\text{if}\; \sum_{n\geq 1} f\big(\rho^n\varphi(n)\big)\cdot \rho^{-\gamma n}=\infty.
  \end{array}
\right.$
\end{thm}

\begin{rem}
  \emph{(i)} In Theorem \ref{dict-lebesgue}, the dichotomy law still holds if the strong separation condition is replaced by the open set condition. However, by our approach, this condition could not be relaxed in Theorem \ref{dict-hausdorff}. \par
  \emph{(ii)} If we apply the last two theorems to $\mathcal{I}=\{\frac{x}{3}, \frac{x}{3}+\frac{2}{3}\}$, then $T$ is exactly the $\times 3$ map on the middle third Cantor set, and the set $R(\varphi)$ can be written as
  \begin{equation*}
      R(\varphi)=\{x\in K: \|(3^n-1)x\|<\varphi(n) \; \text{for infinitely many} \; n\in\mathbb{N}\},
  \end{equation*}
where $\|y\|$ stands for the distance from $y$ to its nearest integer. Therefore, Theorem \ref{dict-lebesgue} and \ref{dict-hausdorff} correspond to a version of Khintchine's theorem and Jarn\'{i}k's theorem on the middle third Cantor set respectively \emph{(}see \cite{BRV16} for more backgrounds\emph{)}, which provide an answer to Mahler's problem on Diophantine approximation on Cantor set \cite{Mah84}. One can refer to \cite{LSV07} for a similar answer, where the points in $K$ are approximated by rational numbers $p/q$ with $q\in\{3^n: n\in\mathbb{N}\}$. Our results indicate that one can also approximate points in $K$ by
rational numbers with periodic 3-adic expansion.
\end{rem}

We give several applications of the last two theorems. Theorem \ref{Bosh-thm} shows that, for $\mu$-almost every $x\in K$, \eqref{Bosh-result-2} holds for any $\alpha>\gamma$. We also characterize the recurrence behaviors when $\alpha\leq\gamma$, which together with \eqref{Bosh-result-2} sum up to the following $0$-$\infty$ law.

\begin{cor}\label{comp-bosh}
    Let $(K, \mathcal{B}, \mu, |\cdot|, T)$ be as above. For $\mu$-almost every $x\in K$, we have \par
    $$\liminf_{n\rightarrow\infty} n^{\frac{1}{\alpha}} |T^n x-x|=\left\{
  \begin{array}{ll}
   0, &\text{if}\; \alpha\geq \gamma;\\
   \infty, &\text{if}\; \alpha<\gamma.
  \end{array}
\right.$$
\end{cor}

\begin{proof} The case $\alpha>\gamma$ follows from \eqref{Bosh-result-2}. When $\alpha=\gamma$, we apply the divergent part of Theorem \ref{dict-lebesgue} to $\varphi(n)=\eta\cdot n^{-1/\alpha}$, where $\eta>0$ is arbitrarily small. As a consequence,
$\liminf_{n\rightarrow\infty}n^{1/\alpha}|T^n x-x|\leq\eta$ for $\mu$-almost every $x\in K$. Since $\eta>0$ is arbitrarily small, we
get the desired.

On the other hand, when $\alpha<\gamma$, we apply the convergent part of Theorem \ref{dict-lebesgue} to $\varphi(n)=n^{-\frac{2}{\alpha+\gamma}}$. As a result, for $\mu$-almost every $x\in K$, $n^{1/\alpha}|T^n x-x|\geq n^{\frac{\gamma-\alpha}{\alpha(\alpha+\gamma)}}$ eventually, which implies $\liminf_{n\rightarrow\infty}n^{1/\alpha}|T^n x-x|=\infty$.
\end{proof}

\begin{rem}
   Recall that $\mu=\mathcal{H}^{\gamma}|_K$ and $\underline{d}_{\mu}(x)=\gamma$ for $\mu$-almost every $x\in K$. Thus the last corollary extends Theorem \ref{Bosh-thm} and \ref{Bar-Sau} for the system
   $(K, \mathcal{B}, \mu, |\cdot|, T)$ by characterizing how fast the generic recurrence rate could be.
\end{rem}

\begin{cor}\label{cor-dimension}
  Let $(K, \mathcal{B}, \mu, |\cdot|, T)$ be as above and $\varphi:\mathbb{N}\rightarrow\mathbb{R}^+$. Then
 \begin{equation}\label{dim-recur}
   \dim_H R(\varphi)=\dfrac{\gamma}{1+b} \;\; \text{with} \;\; b=\liminf_{n\rightarrow\infty} \dfrac{\log_{\rho} \varphi(n)}{n},
 \end{equation}
where $\dim_H$ denotes the Hausdorff dimension of a set.
\end{cor}

\begin{proof}
It suffices to show that for any $\eta>0$,
\begin{equation*}
  \mathcal{H}^{\frac{\gamma (1-\eta)}{1+b}}(R(\varphi))=\infty\;\text{and}\; \mathcal{H}^{\frac{\gamma (1+\eta)}{1+b}}(R(\varphi))=0.
\end{equation*}
Indeed, by the definition of $b=\liminf_{n\rightarrow\infty} \frac{\log_{\rho} \varphi(n)}{n}$, it is easy to see that
$\varphi^{1+\eta}(n)\leq \rho^{n(b-0.5\eta)}$ for $n$ large enough and
$\varphi^{1-\eta}(n)\geq \rho^{n(b+\eta)}$ for infinitely many $n$.
Letting $f_1(r)=\frac{\gamma (1-\eta)}{1+b}$ and $f_2(r)=\frac{\gamma (1+\eta)}{1+b}$ and applying Theorem \ref{dict-hausdorff}, we get the desired results.
\end{proof}

Corollary \ref{cor-dimension} shows that there are three possibilities for the Hausdorff dimension of $R(\varphi)$ in accordance with the decay rate of $\varphi$:
$$\dim_H R(\varphi)=\left\{
  \begin{array}{ll}
   \gamma, &\text{if}\; \varphi(n)=\rho^{o(n)};\\
   \frac{\gamma}{1+b}, &\text{if}\; \varphi(n)=\rho^{bn+o(n)}\;\text{for some}\;b\in(0, \infty);\\
   0, &\text{otherwise}.
  \end{array}
\right.$$

Further, Theorem \ref{dict-hausdorff} allows us to distinguish those $R(\varphi)$ with equal Hausdorff dimension. For instance, set
$\varphi(n)=\rho^{b n}$ and $\varphi_{\eta}(n)=\rho^{b n}(\log \rho^n)^{\frac{(1+b)(1+\eta)}{\gamma}}$ for some $b>0$ and $\eta>0$. Although it follows immediately from Corollary \ref{cor-dimension} that $\dim_H R(\varphi)=\dim_H R(\varphi_{\eta})=\frac{\gamma}{1+b}$, we have the following exact logarithmic order for recurrence rates.
\begin{cor}\label{log-order}
  Let $(K, \mathcal{B}, \mu, |\cdot|, T)$ be as above. Let $f(r)=r^{\frac{\gamma}{1+b}}$. For any $\eta>0$,
\begin{equation*}
  \mathcal{H}^f (R(\varphi))=\infty\;\;\text{and}\;\; \mathcal{H}^f (R(\varphi_{\eta}))=0.
\end{equation*}
In particular, the set $R(\varphi)\setminus R(\varphi_{\eta})$ is uncountable.
\end{cor}

\begin{proof}
Applying Theorem \ref{dict-hausdorff} to $\varphi(n)=\rho^{b n}$ and
$\varphi_{\epsilon}(n)=\rho^{b n}(\log \rho^n)^{\frac{(1+b)(1+\epsilon)}{\gamma}}$ respectively, the result follows immediately.
\end{proof}

The paper is organized as follows. In section 2, we give some notations and list some known results which are needed in the proof of our main results. In section 3, we prove our main results Theorem \ref{dict-lebesgue} and \ref{dict-hausdorff}.

\section{Preliminaries}
In this section, we first give the definitions of Hausdorff measures and Hausdorff dimensions. Then we list several known results which are useful in the proofs of Theorem \ref{dict-lebesgue} and \ref{dict-hausdorff}. Throughout, $(X, d)$ is assumed to be a separable metric space and $B(x,r)$ denotes the open ball centred at $x\in X$ with radius $r>0$.

Suppose $F\subset X$ is a subset of $X$.  For each $\tau>0$, a countable collection $\{B(x_i, r_i)\}_{i\geq 1}$ of open balls in $X$ is said to be a \emph{$\tau$-cover} of $F$ if $r_i\leq\tau$ for each $i\geq 1$ and $F\subset\bigcup_{i\geq 1}B(x_i, r_i)$. A function $f: \mathbb{R}^+\rightarrow\mathbb{R}^+$ is called a \emph{dimension function} if $f$ is continuous, non-decreasing and satisfies $f(r)\rightarrow 0$ as $r\rightarrow 0$. A dimension function $f$ is \emph{doubling} if there exists $\lambda\geq 1$ such that $f(2r)\leq \lambda f(r)$ for all small $r$.
The \emph{Hausdorff $f$-measure} of $F$ with respect to the dimension function $f$ is defined as
\begin{equation*}
   \mathcal{H}^f(F):=\lim_{\rho\rightarrow 0}\inf\Big\{\sum_{i\geq 1}f(2 r_i): \{B(x_i, r_i)\}_{i\geq 1}\;\text{is a}\; \rho\text{-cover of}\; F\Big\}.
\end{equation*}

When $f(r)=r^s$ for some $s\geq 0$, $\mathcal{H}^f$ is the usual $s$-dimensional Hausdorff measure $\mathcal{H}^s$ and the Hausdorff dimension
of $F$ is 
\begin{equation*}
   \dim_H F:=\inf\big\{s\geq 0: \mathcal{H}^s(F)=0\big\}=\sup\big\{s\geq 0: \mathcal{H}^s(F)=\infty\big\}.
\end{equation*}
For more details, see for example \cite{Fal90, Mat95}.

Further, suppose that $(X, d)$ is locally compact and there exist constants $\delta>0, 0<c_1<1<c_2<\infty$ and $r_0>0$ such that
\begin{equation}\label{ahlfors-reg}
   c_1 r^{\delta}\leq\mathcal{H}^{\delta}(B(x, r))\leq c_2 r^{\delta},
\end{equation}
for any $x\in X$ and $0<r<r_0$. An implication of \eqref{ahlfors-reg} is that $0<\mathcal{H}^{\delta}(X)<\infty$. Thus $\dim_H X=\delta$. The mass transference principle, developed by Beresnevich and Velani \cite{BV06}, builds up a bridge from $\mathcal{H}^{\delta}$-measure theoretic statements for
a limsup subset of $X$ to general $\mathcal{H}^f$-measure theoretic statements. Given a dimension function $f$ and an open ball $B=B(x, r)$, we define
$B^f:=B(x, f(r)^{1/\delta})$.
\begin{thm}\label{MTP}\emph{(Mass transference principle \cite{BV06})}
  Let $(X, d)$ be as above and $\{B_i\}_{i\geq 1}$ be a sequence of balls in $X$ with radius $r(B_i)\rightarrow 0$ as $i\rightarrow\infty$. Let $f$
be a dimension function such that $r^{-\delta}f(r)$ is monotonically increasing as $r\rightarrow 0$. For any ball $B\subset X$ with $\mathcal{H}^{\delta}(B)>0$, if
\begin{equation*}
  \mathcal{H}^{\delta}(B\cap\limsup_{i\rightarrow\infty}B_i^f)=\mathcal{H}^{\delta}(B),
\end{equation*}
then
\begin{equation*}
  \mathcal{H}^f(B\cap\limsup_{i\rightarrow\infty}B_i)=\mathcal{H}^f(B).
\end{equation*}
\end{thm}

In this paper, we will take $X$ to be the homogeneous self-similar $K$ defined in Section $1$. Recall that the natural measure $\mu$ on $K$ is the same as $\mathcal{H}^{\gamma}|_K$. So the formula \eqref{ahlfors-reg} is satisfied for $\mu$, i.e.,
\begin{equation}\label{cantor-alf-reg}
  c_1 r^{\gamma} \leq \mu(B(x, r))\leq c_2 r^{\gamma}.
\end{equation}

Recall that a finite measure $\nu$ on $X$ is said to be doubling if there exists a constant $C\geq 1$ such that for any $x\in X$,
\begin{equation*}
   \nu(B(x, 2r))\leq C \nu(B(x, r)).
\end{equation*}
Here are two measure theoretic results: one for the full measure sets and the other for the measure of a limsup set.
\begin{lem}\label{full-measure}\emph{(see \cite{BDV06}, Sect. 8, Proposition 1)}
   Let $(X, d)$ be a metric space and $\mu$  be a finite doubling measure on $X$ such that any open set is $\mu$ measurable. Let $E$ be a Borel subset of $X$.
Assume that there are constants $r_0, c>0$ such that for any ball $B=B(x, r)$ with $x\in X$ and $r<r_0$, we have
\begin{equation*}
   \mu(E\cap B)\geq c\mu(B).
\end{equation*}
Then $E$ has full measure in $X$, that is, $\mu(X\setminus E)=0$.
\end{lem}

\begin{lem}\label{mea-limsup}\emph{(see \cite{Spr79}, Lemma 5 or \cite{Yan06}, Theorem 1)}
   Let $(X, \mathcal{B}, \mu)$ be a measure space and $\{A_n\}_{n\geq 1}$ be a sequence of measurable sets such that $\sum_{n\geq 1}\mu(A_n)=\infty$. Then
\begin{equation*}
   \mu(\limsup_{n\rightarrow\infty}A_n)\geq\limsup_{N\rightarrow\infty}\dfrac{\big(\sum_{n=1}^N\mu(A_n)\big)^2}{\sum_{m, n=1}^N\mu(A_m\cap A_n)}
   =\limsup_{N\rightarrow\infty}\dfrac{\sum_{1\leq m<n\leq N}\mu(A_m)\mu(A_n)}{\sum_{1\leq m<n\leq N}\mu(A_m\cap A_n)}.
\end{equation*}
\end{lem}

\section{Proofs of the main results}
For each finite word $\epsilon=(\epsilon_1,\cdots, \epsilon_n)\in \Lambda^n$, let
$I(\epsilon_1,\cdots, \epsilon_n):=\phi_{\epsilon}([0, 1])$. We call $I(\epsilon_1,\cdots, \epsilon_n)$ a cylinder of order $n$ (with respect to the self-similar set $K$), which is of length $\rho^n$ and $\mu$-measure $L^{-n}$. Then the limsup
set $R(\varphi)$ defined in \eqref{recurrent} can be written as
\begin{align}\label{limsup-deco}
 R(\varphi)=& \mathop{\bigcap}\limits_{k=1}^{\infty}\mathop{\bigcup}\limits_{n=k}^{\infty}\big\{x\in K: \; |T^n x-x|<\varphi(n)\big\}\notag\\
           =& \mathop{\bigcap}\limits_{k=1}^{\infty}\mathop{\bigcup}\limits_{n=k}^{\infty}\mathop{\bigsqcup}\limits_{\epsilon\in\Lambda^n}
              \big\{x\in I(\epsilon_1,\cdots, \epsilon_n):\; |T^n x-x|<\varphi(n)\big\}\cap K\notag\\
          =& \mathop{\bigcap}\limits_{k=1}^{\infty}\mathop{\bigcup}\limits_{n=k}^{\infty}\mathop{\bigsqcup}\limits_{\epsilon\in\Lambda^n}
              J(\epsilon_1,\cdots, \epsilon_n)\cap K
\end{align}
where \(J(\epsilon_1,\cdots, \epsilon_n):=\big\{x\in I(\epsilon_1,\cdots, \epsilon_n):\; |T^n x-x|<\varphi(n)\big\}.\)

In the following part, the notations $\lesssim$ or $\gtrsim$ will be used to indicate an inequality with an unspecified positive constant.
\subsection{Proof of Theorem \ref{dict-lebesgue}}
Without loss of generality, we may assume $\varphi(n)\leq \frac{1-\rho}{4}$ for all $n\geq 1$. Indeed, suppose this is not the case and define $\psi(n)=\min\{\varphi(n), \frac{1-\rho}{4}\}$ for each $n$. If the series $\sum_{n\geq 1} \varphi^{\gamma}(n)$ converges, then so does $\sum_{n\geq 1} \psi^{\gamma}(n)$. Also, it follows that $\varphi(n)\leq \frac{1-\rho}{4}$ eventually and thus $R(\varphi)=R(\psi)$. On the other hand, if the series $\sum_{n\geq 1} \varphi^{\gamma}(n)$ diverges, then it can be easily verified that $\sum_{n\geq 1} \psi^{\gamma}(n)$ diverges too. Moreover, $R(\psi)\subset R(\varphi)$, so it suffices to show the divergent part for $\psi$. \par

Notice that for each finite word $\epsilon\in\Lambda^n$, any point $x$ in $I(\epsilon_1,\cdots, \epsilon_n)$ can be written as
\begin{equation*}
   x=\sum_{i=1}^n a_{\epsilon_i}\cdot\rho^{i-1} + \rho^n \cdot T^n x=:[\epsilon] + \rho^n \cdot T^n x,
\end{equation*}
thus $|T^n x-x|<\varphi(n)$ implies
$x\in \big(\frac{[\epsilon]}{1-\rho^n}-\frac{\rho^n \varphi(n)}{1-\rho^n}, \frac{[\epsilon]}{1-\rho^n}+\frac{\rho^n \varphi(n)}{1-\rho^n}\big)$.
One can easily check that $J(\epsilon_1,\cdots, \epsilon_n)\neq\emptyset$. Moreover, $J(\epsilon_1,\cdots, \epsilon_n)$ is an interval satisfying
\begin{equation*}
   J(\epsilon_1,\cdots, \epsilon_n)=I(\epsilon_1,\cdots, \epsilon_n)\bigcap\Big(\frac{[\epsilon]}{1-\rho^n}-\frac{\rho^n \varphi(n)}{1-\rho^n}, \frac{[\epsilon]}{1-\rho^n}+\frac{\rho^n \varphi(n)}{1-\rho^n}\Big).
\end{equation*}
By the assumption $\varphi(n)\leq \frac{1-\rho}{4}$, one can show that the length of $J(\epsilon_1,\cdots, \epsilon_n)$ satisfies
\begin{equation}\label{length}
  \rho^n \varphi(n)\leq |J(\epsilon_1,\cdots, \epsilon_n)|\leq \rho^{n-1} \varphi(n).
\end{equation}
Combining \eqref{cantor-alf-reg} and the last formula, we have
\begin{equation}\label{length-J}
 \mu(J(\epsilon_1,\cdots, \epsilon_n))\lesssim (\rho^n\varphi(n))^{\gamma}.
\end{equation}

\medskip
\textbf{The convergent part.}
By \eqref{limsup-deco} and \eqref{length-J},
\begin{equation*}
 \begin{split}
   \mu(R(\varphi(n)))\leq &\liminf_{k\rightarrow\infty}\sum_{n=k}^{\infty}\sum_{\epsilon\in\Lambda^n}\mu(J(\epsilon_1,\cdots,\epsilon_n))\\
                     \lesssim & \liminf_{k\rightarrow\infty}\sum_{n=k}^{\infty} (\rho^n\varphi(n))^{\gamma}\cdot L^n \\
                     = & \liminf_{k\rightarrow\infty}\sum_{n=k}^{\infty} \varphi^{\gamma}(n)=0.
 \end{split}
\end{equation*}

\textbf{The divergent part.} By Lemma \ref{full-measure}, it suffices to show that for any ball $B=B(x, r)$ with $x\in K$ and $r$ small enough, we have
\begin{equation}\label{loc-positive-mea}
   \mu(R(\varphi)\cap B)\geq c \mu(B)
\end{equation}
for some positive constant $c$ independent of $B$. Here we take $B=B(x, r)$ to be an arbitrary ball with $x\in K$ and $\mu(2B):=\mu(B(x, 2r))$ satisfying \eqref{cantor-alf-reg}. Let $n_0:=n_0(B)$ be a sufficiently large integer so that $\rho^{n_0}<r(B)$. For $n\geq 1$, let
\begin{equation*}
   A_n(B)=\{x\in K:\; |T^n x-x|<\varphi(n)\}\cap B=\bigsqcup_{{\epsilon\in\Lambda^n}}J(\epsilon_1,\cdots,\epsilon_n)\cap K\cap B.
\end{equation*}
Then
\begin{equation*}
   R(\varphi)\cap B=\limsup_{n\rightarrow\infty} A_n(B).
\end{equation*}

For $n\geq n_0$, we have
\begin{equation*}
    \#\{\epsilon\in\Lambda^n:\; J(\epsilon_1,\cdots,\epsilon_n)\subset B\}\geq
    \#\{\epsilon\in\Lambda^n:\; I(\epsilon_1,\cdots,\epsilon_n)\subset B\}\geq \frac{\mu(B)}{L^{-n}}.
\end{equation*}
It follows that
\begin{equation}\label{An-mea-lb}
   \mu(A_n(B))\gtrsim \frac{\mu(B)}{L^{-n}}\cdot (\rho^n\varphi(n))^{\gamma}=\mu(B)\varphi^{\gamma}(n),
\end{equation}
which together with the assumption $\sum_{n\geq 1}\varphi^{\gamma}(n)=\infty$ implies
\begin{equation}\label{local-inf-mea}
  \sum_{n\geq 1}\mu(A_n(B))=\infty.
\end{equation}

The key ingredient to establish \eqref{loc-positive-mea} and thus the divergent part of Theorem \ref{dict-lebesgue} is the following  `local quasi-independence on average' result.
\begin{lem}\emph{(Local quasi-independence on average)}\label{local-quasi-indep}
  There exists a constant $C>1$ such that for any ball $B$ as above and for $N$ sufficiently large,
\begin{equation}\label{form-local}
  \sum_{m, n=1}^N\mu(A_m(B)\cap A_n(B))\leq \frac{C}{\mu(B)}\Big(\sum_{n=1}^N\mu(A_n(B))\Big)^2.
\end{equation}
\end{lem}
Indeed, \eqref{local-inf-mea} and Lemma \ref{local-quasi-indep} together with Lemma \ref{mea-limsup} imply \eqref{loc-positive-mea}. This completes the proof of the divergent part upon the `local quasi-independence on average' result, which we will prove now.

\medskip
\begin{proof}[Proof of Lemma \ref{local-quasi-indep}] Recall that $B$ is a fixed ball centred at a point in $K$ such that $\mu(2B)$ satisfies \eqref{cantor-alf-reg} and $n_0:=n_0(B)$ is chosen so that $\rho^{n_0}<r(B)$. For any $n>m>n_0$,
\begin{align*}
    \mu(A_m(B)\cap A_n(B))= & \mu\Big(\bigsqcup_{{\epsilon\in\Lambda^m}}J(\epsilon_1,\cdots,\epsilon_m)\cap K\cap B\cap A_n(B)\Big)\notag \\
                          = & \sum_{\epsilon\in\Lambda^m}\mu(J(\epsilon_1,\cdots,\epsilon_m)\cap B\cap A_n(B)) \notag \\
                       \lesssim & \;  \mathcal{N}(m, B)\cdot\max_{\epsilon\in\Lambda^m}\mu(J(\epsilon_1,\cdots,\epsilon_m)\cap A_n(B)),
\end{align*}
where
\begin{align*}
   \mathcal{N}(m, B)=  & \#\{\epsilon\in\Lambda^m:\; J(\epsilon_1,\cdots,\epsilon_m)\cap B\neq\emptyset\}\notag\\
                   \leq & \#\{\epsilon\in\Lambda^m:\; I(\epsilon_1,\cdots,\epsilon_m)\cap B\neq\emptyset\}\notag\\
                   \leq & \#\{\epsilon\in\Lambda^m:\; I(\epsilon_1,\cdots,\epsilon_m)\subset 2B\}
                   \lesssim \frac{\mu(B)}{L^{-m}}.
\end{align*}
Thus,
\begin{equation}\label{AmAn}
   \mu(A_m(B)\cap A_n(B))\lesssim \frac{\mu(B)}{L^{-m}}\cdot\max_{\epsilon\in\Lambda^m}\mu(J(\epsilon_1,\cdots,\epsilon_m)\cap A_n(B)).
\end{equation}
Next, we shall obtain an upper bound for $\max_{\epsilon\in\Lambda^m}\mu(J(\epsilon_1,\cdots,\epsilon_m)\cap A_n(B))$. Indeed, for any fixed finite word $\epsilon\in\Lambda^m$,
\begin{align}\label{JmCapAn}
     &\mu(J(\epsilon_1,\cdots,\epsilon_m)\cap A_n(B))\notag\\
   = & \mu\Big(J(\epsilon_1,\cdots,\epsilon_m)\cap \big(\bigsqcup_{\epsilon_{m+1}, \cdots, \epsilon_n\in\Lambda} J(\epsilon_1,\cdots,\epsilon_m,\epsilon_{m+1},\cdots,\epsilon_n)\cap K\cap B\big)\Big)\notag\\
   \leq & \sum_{\epsilon_{m+1}, \cdots, \epsilon_n\in\Lambda}\mu(J(\epsilon_1,\cdots,\epsilon_m)\cap J(\epsilon_1,\cdots,\epsilon_m,\epsilon_{m+1},\cdots,\epsilon_n)).
\end{align}
We consider two cases depending on the size of $\frac{2\rho^m \varphi(m)}{1-\rho^m}$ compared to $\rho^n$.

Case (i): $n>m>n_0$ such that $\rho^n\geq \frac{2\rho^m \varphi(m)}{1-\rho^m}$. Then for each $\epsilon\in\Lambda^n$,
$|I(\epsilon_1,\cdots,\epsilon_n)|=\rho^n\geq \frac{2\rho^m \varphi(m)}{1-\rho^m}\geq |J(\epsilon_1,\cdots,\epsilon_m)|$. As a result, each
$J(\epsilon_1,\cdots,\epsilon_m)$ can intersect at most two $J(\epsilon_1,\cdots,\epsilon_m,\epsilon_{m+1},\cdots,\epsilon_n)$.  By \eqref{length-J} and  \eqref{JmCapAn}, for any $\epsilon\in\Lambda^m$, $\mu(J(\epsilon_1,\cdots,\epsilon_m)\cap A_n(B))\lesssim (\rho^n\varphi(n))^{\gamma}$. So \eqref{AmAn} yields
\begin{equation}\label{case-1}
  \mu(A_m(B)\cap A_n(B))
  \lesssim \; \frac{\mu(B)}{L^{-m}}\cdot (\rho^n\varphi(n))^{\gamma}.
\end{equation}

\smallskip
Case (ii): $n>m>n_0$ such that $\rho^n< \frac{2\rho^m \varphi(m)}{1-\rho^m}$. It follows from \eqref{JmCapAn} that
\begin{align*}
   & \; \mu(J(\epsilon_1,\cdots,\epsilon_m)\cap  A_n(B)) \\
     \lesssim \;
    & \mathcal{N}(n, J(\epsilon_1,\cdots,\epsilon_m)) \max_{\epsilon_{m+1}, \cdots, \epsilon_n\in\Lambda}
   \mu(J(\epsilon_1,\cdots,\epsilon_m,\epsilon_{m+1},\cdots,\epsilon_n))
\end{align*}
where
\begin{align*}
   \mathcal{N}(n, J(\epsilon_1,\cdots,\epsilon_m))=
   & \#\{\epsilon\in\Lambda^n:\; J(\epsilon_1,\cdots,\epsilon_m,\epsilon_{m+1},\cdots,\epsilon_n)\cap
      J(\epsilon_1,\cdots,\epsilon_m)\neq\emptyset\}\\
   \leq & \#\{\epsilon\in\Lambda^n:\; I(\epsilon_1,\cdots,\epsilon_m,\epsilon_{m+1},\cdots,\epsilon_n)\cap
      J(\epsilon_1,\cdots,\epsilon_m)\neq\emptyset\} \\
   \leq & \#\{\epsilon\in\Lambda^n:\; I(\epsilon_1,\cdots,\epsilon_m,\epsilon_{m+1},\cdots,\epsilon_n)\subset
      3 J(\epsilon_1,\cdots,\epsilon_m)\} \\
   \lesssim & \; \frac{\varphi^{\gamma}(m)\cdot \rho^{m\gamma}}{L^{-n}}.
\end{align*}
Since $\mu(J(\epsilon_1,\cdots, \epsilon_n))\lesssim (\rho^n\varphi(n))^{\gamma}$ for all $\epsilon\in\Lambda^n$, it follows that
\begin{equation}\label{JmAn}
    \max_{\epsilon\in\Lambda^m}\mu(J(\epsilon_1,\cdots,\epsilon_m)\cap A_n(B))
    \lesssim \frac{\varphi^{\gamma}(m)\cdot \rho^{m\gamma}}{L^{-n}}\cdot (\rho^n\varphi(n))^{\gamma}.
\end{equation}
Therefore,
\begin{align}\label{case-2}
  \mu(A_m(B)\cap A_n(B))
  \lesssim &\; \frac{\mu(B)}{L^{-m}}\cdot \frac{\varphi^{\gamma}(m)\cdot \rho^{m\gamma}}{L^{-n}}\cdot (\rho^n\varphi(n))^{\gamma} && (\text{by \eqref{AmAn}}) \notag\\
  \lesssim &\; \mu(B)\cdot\varphi^{\gamma}(m)\cdot\varphi^{\gamma}(n) && (\text{since}\; \rho^{\gamma}=L^{-1})\notag\\
  \lesssim &\; \frac{1}{\mu(B)}\mu(A_m(B))\mu(A_n(B)). && \text{(by \eqref{An-mea-lb})}
\end{align}

Now we prove \eqref{form-local}. For $N$ sufficiently large,
\begin{eqnarray*}
     &&\sum_{m, n=1}^N\mu(A_m(B)\cap A_n(B)) \\
     &=\;& \sum_{n=1}^N \mu(A_n(B)) + 2 \sum_{1\leq m<n\leq N}\mu(A_m(B)\cap A_n(B))\\
     &=\;& \sum_{n=1}^N \mu(A_n(B)) +2 \left[\sum_{1\leq m<n\leq n_0\leq N}\mu(A_m(B)\cap A_n(B)) \right.\\
         && \left.+ \sum_{1\leq m< n_0< n\leq N}\mu(A_m(B)\cap A_n(B)) + \sum_{n_0\leq m<n\leq N}\mu(A_m(B)\cap A_n(B))\right]\\
     &=:&   S_1+ 2(S_2+S_3+S_4).
\end{eqnarray*}
By \eqref{local-inf-mea}, one can check that $S_1$, $S_2$ and $S_3$ are all less than $\frac{1}{\mu(B)}\big(\sum_{n=1}^N\mu(A_n(B))\big)^2$ for $N$ sufficiently large. Further,
\begin{align*}
  S_4=\sum_{\substack{n_0\leq m<n\leq N\\ \text{Case (i)}}}\mu(A_m(B)\cap A_n(B))+\sum_{\substack{n_0\leq m<n\leq N\\ \text{Case (ii)}}}\mu(A_m(B)\cap A_n(B))
\end{align*}
Since in Case (i), $\rho^n\geq \frac{2\rho^m \varphi(m)}{1-\rho^m}$, we have
\begin{align*}
    \sum_{\substack{n_0\leq m<n\leq N\\ \text{Case (i)}}}\mu(A_m(B)\cap A_n(B))
    \leq & \sum_{m=n_0}^N\sum_{n=m+1}^{n_m\wedge N} \mu(A_m(B)\cap A_n(B)) && \\
    \lesssim & \sum_{m=n_0}^N\sum_{n=m+1}^{n_m\wedge N}\frac{\mu(B)}{L^{-m}}(\rho^n\varphi(n))^{\gamma} && \text{by \eqref{case-1}} \\
      \leq & \sum_{m=n_0}^N\sum_{n=m+1}^{n_m\wedge N} \frac{\mu(B)\varphi^{\gamma}(n)}{L^{n-m}} && \text{since}\; \rho^{\gamma}=L^{-1}\\
      \lesssim & \sum_{n=n_0+1}^N \mu(A_n(B))\sum_{k=1}^{N-n_0} \frac{1}{L^k} && \text{by \eqref{An-mea-lb}}\\
      \lesssim & \; \frac{1}{\mu(B)}\Big(\sum_{n=1}^N\mu(A_n(B))\Big)^2 && \text{by \eqref{local-inf-mea}}
\end{align*}
for $N$ sufficiently large, where $n_m=m+\big[\log_3 \frac{1}{2\varphi(m)}\big]+1$. Finally, it follows from \eqref{case-2} that
\begin{equation*}
  \sum_{\substack{n_0\leq m<n\leq N\\ \text{Case (ii)}}}\mu(A_m(B)\cap A_n(B))\lesssim \; \frac{1}{\mu(B)}\Big(\sum_{n=1}^N\mu(A_n(B))\Big)^2
\end{equation*}
for $N$ sufficiently large. Therefore, we arrive at the required. \qedhere
\end{proof}

\subsection{Proof of Theorem \ref{dict-hausdorff}}
\
\smallskip
\par
\textbf{The convergent part.} Recall that $f$ is a doubling dimension function and each $J(\epsilon_1,\cdots,\epsilon_n)$ is an interval whose length satisfies \eqref{length}. By \eqref{limsup-deco} and the definition of Hausdorff $f$-measure,
\begin{equation*}
  \begin{split}
    \mathcal{H}^f(R(\varphi))\leq &\liminf_{k\rightarrow\infty}\sum_{n=k}^{\infty}\sum_{\epsilon\in\Lambda^n} f(\rho^{n-1}\varphi(n))\\
    \lesssim &\; \liminf_{k\rightarrow\infty}\sum_{n=k}^{\infty} f(\rho^n\varphi(n))\cdot (\rho^{-\gamma})^n=0.
  \end{split}
\end{equation*}
Thus, $\mathcal{H}^f(R(\varphi))=0$ as required.

\smallskip
\textbf{The divergent part.} By the strong separation condition, the minimal gap between distinct cylinders of order $1$ is
$\min_{1\leq j\leq L-1}(a_{j+1}-a_j-\rho)>0$. Without loss of generality, we assume that $\varphi(n)\rightarrow 0$ as $n\rightarrow\infty$ (otherwise, $R(\varphi)=K$ and the assertion is obvious). It follows that $\varphi(n)\leq\frac{1}{4}(1-\rho)\min_{1\leq j\leq L-1}(a_{j+1}-a_j-\rho)$ for all sufficiently large $n$ and for any $\epsilon\in\Lambda^n$, the ball $B(\frac{[\epsilon]}{1-\rho^n}, \frac{\rho^n \varphi(n)}{1-\rho^n})$ intersects only one cylinder $I(\epsilon_1, \cdots, \epsilon_n)$ of order $n$. Since $K=\mathop{\bigcap}\limits_{k=1}^{\infty}\mathop{\bigcup}\limits_{n=k}^{\infty}\mathop{\bigsqcup}\limits_{\epsilon\in\Lambda^n}I(\epsilon_1,\cdots,\epsilon_n)$,
we have
\begin{equation}\label{mu-1}
  J(\epsilon_1, \cdots, \epsilon_n)\cap K=B\Big(\frac{[\epsilon]}{1-\rho^n}, \frac{\rho^n \varphi(n)}{1-\rho^n}\Big)\cap K.
\end{equation}
Let $$\widetilde{\varphi}(n)=f^{1/\gamma}\big(\frac{\rho^n \varphi(n)}{1-\rho^n}\big)\times \frac{1-\rho^n}{\rho^n}$$
and
$$\widetilde{J}(\epsilon_1, \cdots, \epsilon_n)=I(\epsilon_1, \cdots, \epsilon_n)\cap B\Big(\frac{[\epsilon]}{1-\rho^n}, \frac{\rho^n \widetilde{\varphi}(n)}{1-\rho^n}\Big).$$
Obviously,
\begin{equation}\label{mu-2}
  \widetilde{J}(\epsilon_1, \cdots, \epsilon_n)\cap K \subset
   B\Big(\frac{[\epsilon]}{1-\rho^n}, \frac{\rho^n \widetilde{\varphi}(n)}{1-\rho^n}\Big)\cap K.
\end{equation}
Moreover, by the monotonicity of $f$ and the assumption,
\begin{equation*}
  \sum_{n=1}^{\infty}\widetilde{\varphi}^{\gamma}(n)
  \gtrsim\sum_{n=1}^{\infty}f\big(\rho^n\varphi(n)\big)\cdot \rho^{-\gamma n}=\infty.
\end{equation*}
It follows from the divergent part of Theorem \ref{dict-lebesgue} that
\begin{equation*}
  \mathcal{H}^{\gamma}\big(\mathop{\bigcap}\limits_{k=1}^{\infty}\mathop{\bigcup}\limits_{n=k}^{\infty}\mathop{\bigsqcup}
  \limits_{\epsilon\in\Lambda^n}\widetilde{J}(\epsilon_1, \cdots, \epsilon_n)\cap K\big)=\mathcal{H}^{\gamma}(R(\widetilde{\varphi})
  \cap K)=\mathcal{H}^{\gamma}(K).
\end{equation*}
The above equation and \eqref{mu-2} yield
\begin{align*}
  \mathcal{H}^{\gamma}(K)
=&\mathcal{H}^{\gamma}\left(\mathop{\bigcap}\limits_{k=1}^{\infty}\mathop{\bigcup}\limits_{n=k}^{\infty}\mathop{\bigcup}
  \limits_{\epsilon\in\Lambda^n}B\Big(\frac{[\epsilon]}{1-\rho^n}, \frac{\rho^n \widetilde{\varphi}(n)}{1-\rho^n}\Big)\cap K\right)\\
=&\mathcal{H}^{\gamma}\left(\mathop{\bigcap}\limits_{k=1}^{\infty}\mathop{\bigcup}\limits_{n=k}^{\infty}\mathop{\bigcup}
  \limits_{\epsilon\in\Lambda^n}B\Big(\frac{[\epsilon]}{1-\rho^n}, f^{1/\gamma}\big(\frac{\rho^n \varphi(n)}{1-\rho^n}\big)\Big)\cap K\right).
\end{align*}
Applying Theorem \ref{MTP} to the collection of balls
$\big\{B\big(\frac{[\epsilon]}{1-\rho^n}, \frac{\rho^n \varphi(n)}{1-\rho^n}\big)\big\}_{n\geq 1, \epsilon\in\Lambda^n}$, we have
\begin{equation*}
   \mathcal{H}^f\Big(\mathop{\bigcap}\limits_{k=1}^{\infty}\mathop{\bigcup}\limits_{n=k}^{\infty}\mathop{\bigsqcup}
  \limits_{\epsilon\in\Lambda^n}B\big(\frac{[\epsilon]}{1-\rho^n}, \frac{\rho^n \varphi(n)}{1-\rho^n}\big)\cap K\Big)=\mathcal{H}^f(K),
\end{equation*}
which together with \eqref{mu-1} implies $\mathcal{H}^f(R(\varphi))=\mathcal{H}^f(K)$. The proof is complete.

\section*{Acknowledgements} This work was supported by NSFC (Grant No. 11771153), the Fundamental Research Funds for the Central Universities (No. 2017MS110) and the Characteristic innovation project of colleges and universities in Guangdong (No. 2016KTSCX007).

\bibliographystyle{amsplain}

\end{document}